\documentclass{elsarticle}
\usepackage{amsmath}
\usepackage{amsthm}
\usepackage{mathrsfs}

\usepackage{tikz}
\usepackage{verbatim}
\usepackage{enumitem}
\usepackage{amssymb}
\usetikzlibrary{matrix}
\newtheorem{thm}{Theorem}[section]
\newtheorem{prop}[thm]{Proposition}
\newtheorem{lem}[thm]{Lemma}
\newtheorem{cor}[thm]{Corollary}
\newtheorem{qn}[thm]{Question}

\theoremstyle{definition}
\newtheorem{definition}[thm]{Definition}
\newtheorem{example}[thm]{Example}

\theoremstyle{remark}
\newtheorem{remark}[thm]{Remark}

\numberwithin{equation}{section}

\begin{document}

\title{The radical-annihilator monoid of a ring} 

\author{Ryan C. Schwiebert}

\address{360fly\\ 1000 Town Center Way Ste. 200\\ Pittsburgh, PA 15317}
\address{Center of Ring Theory and its Applications\\ 
321 Morton Hall
Ohio University\\ Athens, OH 45701}
\ead[R.~Schwiebert]{ryan.c.schwiebert@gmail.com}

\begin{abstract}
  Kuratowski's closure-complement problem gives rise to a monoid generated by the closure and complement operations. Consideration of this monoid yielded an interesting classification of topological spaces, and subsequent decades saw further exploration using other set operations. This article is an exploration of a natural analogue in ring theory: a monoid produced by ``radical'' and ``annihilator'' maps on the set of ideals of a ring. We succeed in characterizing semiprime rings and commutative dual rings by their radical-annihilator monoids, and we determine the monoids for commutative local zero-dimensional (in the sense of Krull dimension) rings.
\end{abstract}


\maketitle

\section{Introduction}
\label{intro}

Kuratowski famously showed in \cite{KURA22} that given a subset of a topological space, it is possible to make at most $14$ distinct sets from this subset by using the closure and  complement operations. This led to the consideration of the ``Kuratowski monoid'' of a space, where the two operations are looked upon as generators for a monoid of functions on the powerset of the space. The monoid formed this way can only be one of six particular types depending on the topology.

There are actually two problems concealed here. One is ``how many elements does the monoid have for a given space?'' and the other is ``what is the maximum number of distinct sets one can make with a given subset?'' In general, the monoid for a space may be strictly larger than the number of different sets producible from a single input.

In the literature, many variations on Kuratowski's problem have been carried out with other set operations in topological spaces with similar interesting results. For example, the closure and complement operators have also been mixed in various combinations with the interior operator, intersection, and union. Solutions to these variations appear comprehensively in \cite{GardJack08} and \cite{Sherman_2010}, where the authors organize, enhance, and add to earlier work in primary sources including \cite{Jackson_2004}, \cite{Jackson_Stokes_2004}, \cite{Soltan_1982}, \cite{Yu_1978} and \cite{Zarycki_1927}.

Further work has been done using ``generalized closure operators'' which are more general than topological closures in the sense that they do not necessarily distribute over finite unions. This gives rise to a generalized interior and the same problems tackled in the topological case. Early work began in \cite{Hammer_1960} and was followed by \cite{Coleman_1967}, \cite{Shum_1996}, and \cite{Soltan_1981}. More recently, two closure operators with complementation were considered in \cite{Shallit_2011}. Contexts to which this can be applied include formal languages (\cite{Brzozowski_2011}, \cite{Peleg_1984}), operators on binary relations (\cite{Fishburn_1978}, \cite{Graham_1972}), and examples from universal algebra (\cite{Anusiak_1971}, \cite{Comer_1972}, \cite{Nelson_1967}, \cite{Pigozzi_1972}).

In the category of rings, the map sending an ideal to its radical and the map sending an ideal to its annihilator are natural candidates for replacement of closure and complementation. While the radical map qualifies as a generalized closure operator, the work here is not subsumed by the material cited above since the replacement for complementation is not the identity when applied twice. In fact, even when the annihilator map \emph{does} satisfy the condition that two applications are the identity, the annihilator map can still have a different character from complementation: for example, the map can have a fixed point.

Since the annihilator map (and another map called the dualradical map which will be considered later) has properties extending those of the complementation map, it is justifiable to think of it as a ``generalized complementation operator'' just as previous work considered generalized closure operators.
 
This article will investigate the monoids generated by the radical and annihilator operations, will demonstrate differences from the original Kuratowski problem, and will seek to reconcile properties of the monoid with properties of the ring. 

Section \ref{sec:BACKGROUND} is a brief review of the ideas surrounding the Kuratowski monoid and a segue to the radical and annihilator maps. 

Section \ref{sec:PRODUCTS} establishes how to analyze the radical-annihilator monoids of finite direct products of rings. Section \ref{sec:TYPES} determines the monoids of three types of rings: commutative local dual rings, possibly noncommutative semiprime rings and prime rings, and commutative local zero-dimensional rings. The results for the lattermost type cover perfect rings and Artinian rings. 

Section \ref{sec:DUALRAD} studies the properties of the dualradical map, another generalized complement operator. The monoid generated by the radical and dualradical maps is determined for all rings.

Throughout the paper, rings will be assumed to have identity, but will not always be commutative.

\section{Background}\label{sec:BACKGROUND}
\subsection{The Kuratowski monoid}
For a broad look at the Kuratowski $14$-set problem and many related results, see Gardner \& Jackson's article \cite{GardJack08}. There is also an extensive online compilation devoted to Kuratowski-type problems maintained by Bowron \cite{Cornucopia}. 

The (topological) closure and (set) complement operations, denoted here by $k$ and $c$ respectively, act on the subsets of a topological space $X$ partially ordered by inclusion. The two mappings generate a submonoid of the monoid of all mappings from the powerset of $X$ to the powerset of $X$, and this monoid is usually called the $\textbf{Kuratowski monoid}$ of the topological space.

We note the following properties:

\begin{itemize}
  \item If $A\subseteq B$, then $k(A)\subseteq k(B)$; ($k$ is monotone)
  \item $A\subseteq k(A)$; ($k$ is extensive)
  \item $kk=k$; ($k$ is idempotent)
  \item If $A\subseteq B$, then $c(A)\supseteq c(B);$ ($c$ is antitone)
  \item $cc(A)=A$. ($c$ is an involution)
\end{itemize}

Two of the points provide relations in the monoid ($kk=k$ and $cc=1$). Two other points deal with preserving or reversing order. Since these two generators behave this way, we see that the Kuratowski monoid is actually a submonoid of order-reversing-or-preserving poset maps on the powerset of $X$. Combining these properties, it turns out that there is another nontrivial relation in the form
$kckckck=kck$. It is this relation that forces the size of the monoid generated by $c$ and $k$ to remain finite. 

Proving this extra relation involves the monotone and antitone properties of the maps: the algebraic relations alone are not enough. Rather than casting $k$ and $c$ merely as morphisms in the category of sets, it seems more accurate to view them instead as morphisms in a category whose objects are posets and whose morphisms are all order-reversing functions and order-preserving functions between posets. Considering monoids generated by arbitrary set functions is probably too ambitious, and concentrating on the order-reversing and order-preserving ones seems to be a more natural setting for questions like the closure-complement problem.

To visualize the monoid in this category, one can draw Hasse diagrams as in \cite{GardJack08}. The partial order on the elements of the monoid is the natural one on the set of functions between posets: $f\leq g$ means $f(x)\leq g(x)$ for all $x$ in the poset.

\begin{figure}[ht!]
\begin{center}
\begin{tikzpicture}[description/.style={fill=white,inner sep=2pt}]
\matrix (m) [matrix of math nodes, row sep=1.5em,
column sep=0.3em, text height=1.5ex, text depth=0.25ex]
{ &&\pmb{k}&&\ \ &&kc&& \\
  &&\pmb{kckck}&&&&kckckc&&\\
  \pmb{1}&\pmb{ckck}&&\pmb{kckc}&&ckckc&&kck&c \\
  &&\pmb{ckckckc}&&&&ckckck&& \\ 
  &&\pmb{ckc}&&\ \ &&ck&&\\
  &&&&&&&& \\};

\path[-] (m-1-3) edge (m-3-1)
         (m-3-1) edge (m-5-3)
         (m-1-3) edge (m-2-3)
         (m-4-3) edge (m-5-3)
         (m-2-3) edge (m-3-2)
         (m-2-3) edge (m-3-4)
         (m-3-2) edge (m-4-3)
         (m-3-4) edge (m-4-3)
         
         (m-1-7) edge (m-3-9)
         (m-1-7) edge (m-2-7)
         (m-3-9) edge (m-5-7)
         (m-4-7) edge (m-5-7)
         (m-2-7) edge (m-3-6)
         (m-2-7) edge (m-3-8)
         (m-3-6) edge (m-4-7)
         (m-3-8) edge (m-4-7)

;
\end{tikzpicture}
\caption{The full Kuratowski monoid having relations $c^2=1$; $k^2=k$; $kckckck=kck$. The operators in bold are idempotent.}
\label{fig:KURA}
\end{center}
\end{figure}
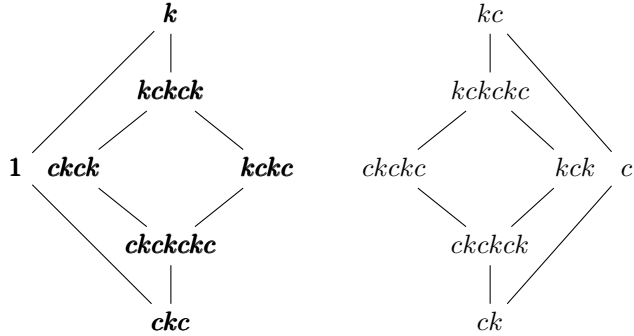

For example, the largest Kuratowski monoid has form depicted in Figure \ref{fig:KURA}. The remaining Kuratowski monoids are collapses of this diagram. There are a total of eight collapses that are algebraically possible, but actually only six of them can arise as the Kuratowski monoid of a topological space. The pair of maps $r$ and $a$ from ring theory suggested below yield new behavior.

\subsection{The radical and annihilator maps}
We assume the reader has knowledge of prime and radical ideals from commutative algebra, and review the noncommutative generalizations here.

\begin{definition}A proper ideal $P$ of a ring $R$ is called a \textbf{prime ideal} if for any pair of ideals $I$, $J$, $IJ\subseteq P$ implies $I\subseteq P$ or $J\subseteq P$. An ideal $Q$ of a ring $R$ is called a \textbf{semiprime ideal} if for any ideals $I$, $I^n\subseteq Q$ for some $n$ implies $I\subseteq Q$. \end{definition}

A nonzero ring is called a \textbf{prime ring} if its zero ideal is a prime ideal, and a ring is called a \textbf{semiprime ring} if its zero ideal is a semiprime ideal. A ring is called a \textbf{reduced ring} if it has no nonzero nilpotent elements. A reduced ring is always semiprime, but not conversely in general. The two notions coincide for commutative rings.

It should be noted that the definition of semiprime ideals does not exclude $R$ from being a semiprime ideal. The relationship between the two types of ideals is as follows.

\begin{prop}The proper semiprime ideals of a ring are exactly the intersections of nonempty sets of prime ideals. \end{prop}

If we consider $R$ to be an empty intersection of prime ideals then we can drop the words ``proper'' and ``nonempty'' from the proposition to make the above statement a complete characterization of semiprime ideals.

\begin{definition}For a proper ideal $I$ of a ring $R$, we define the \textbf{(prime) radical} of $I$ to be the intersection of prime ideals containing $I$, and we denote it by $r(I)$. We let $r(R)=R$ for consistency. (If there is ambiguity about which ring the radical is taken to be in, we include the ring in a subscript of $r$.)\end{definition}

In noncommutative algebra, the radical of the zero ideal of $R$ ($r(\{0\})$ in notation) is also known as the lower nilradical or the Baer-McCoy radical and is denoted $Nil_\ast(R)$. For commutative rings, it is usually called the nilradical of $R$. The above definitions have stretched this usage of ``prime radical'' for use with ideals. 

For expediency, we may use $N(R)$ or simply $N$ to denote the prime radical of a ring. In light of the characterization of semiprime ideals, $r(I)$ is the smallest semiprime ideal containing $I$. In the commutative case, this boils down to the familiar set $\sqrt{I}:=\{r\in R\mid \exists n\in \mathbb{N}, r^n\in I\}$. It is also worth noting for later use that for any ideal $I\subseteq r(\{0\})$, we necessarily have that $r(I)=r(\{0\})$.

The second concept needed is the \emph{annihilator of an ideal}. In general, there is the right annihilator of an ideal $I$ in $R$ defined to be $\{r\in R\mid Ir=\{0\}\}$ and there is its left-hand counterpart. It turns out that the left and right annihilators of an ideal are both also ideals of $R$. However, working with left and right annihilators is unattractive for the purposes of the current article. We will sidestep this complexity by only considering rings in which the left and right annihilator of each ideal are the same set. This holds for commutative rings, of course, but it also holds for semiprime rings.

\begin{lem}(\cite{LAMLMR} Lemma 11.36)
The right and left annihilators of a two-sided ideal in a semiprime ring coincide.
\end{lem}

\begin{definition}
Let $R$ be a ring in which the left and right annihilators of each two-sided ideal coincide. For an ideal $I$, the left (=right) annihilator of $I$ will simply be called the \textbf{annihilator}. The map sending $I$ to its annihilator will be denoted by $a$. (If there is ambiguity about which ring the annihilator is taken to be in, we include the ring in a subscript of $a$.)
\end{definition}

The maps $r$ and $a$ now act on the poset of \emph{ideals} of a ring $R$ partially ordered by inclusion. For comparison with the properties of the closure and complement maps, we list similar properties of the radical and annihilator maps where $A$ and $B$ are ideals:

\begin{itemize}
  \item If $A\subseteq B$, then $r(A)\subseteq r(B)$; ($r$ is monotone)
  \item $A\subseteq r(A)$; ($r$ is extensive)
  \item $rr = r$; ($r$ is idempotent)
  \item If $A\subseteq B$, then $a(A)\supseteq a(B)$; ($a$ is antitone)
  \item $aa(A)\supseteq A$; ($aa$ is extensive)
  \item $aaa(A)=a(A)$, and in particular, $aa$ is idempotent.
\end{itemize}

Since the complement operation satisfies $cc=1$, it is extensive and satisfies $ccc=c$ just as $a$ does above. So, the algebraic relations for the radical-annihilator monoid are a generalization of those in the Kuratowski monoid.

We will call the resulting monoid generated by $a$ and $r$ the \textbf{radical-annihilator monoid} of the ring. We can now pursue an investigation of some rings and their monoids that is analogous to the line of questioning for the Kuratowski monoid. We will adapt the terminology from Gardner and Jackson's paper (\cite{GardJack08} Definition 1.2) to apply to radical-annihilator monoids.

\begin{definition}
In a ring $R$ for which the radical-annihilator monoid is defined,
\begin{enumerate}
 \item the \textbf{$k$-number of an ideal} $I$ of $R$ is defined to be the number of distinct sets obtainable from $I$ under the action of elements of the monoid. (In other words, it's the size of the orbit of $I$ under the action of the monoid.)
 \item the \textbf{$k$-number of $R$} is defined to be the maximum $k$-number among ideals of $R$, and $\infty$ if the set of $k$-numbers of ideals is unbounded.
 \item The \textbf{$K$-number of $R$} is defined to be the cardinality of its radical-annihilator monoid.
\end{enumerate}
\end{definition}

Clearly the $K$-number of $R$ is no less than the $k$-number of $R$, which is in turn no less than the $k$-numbers of each ideal of $R$.

\begin{example}\label{ex:1} Let $F$ be a field and let $R=F[x]/(x^2)$. Then the radical-annihilator monoid for $R$ satisfies the relations $a^2=1$, $r^2=r$, and $rararar=rar$, $rara=rar=arar$. 

The ring has only three ideals, and it is easy to verify that the trivial ideals have $k$-numbers of $3$, and the unique nonzero ideal has a $k$-number of $1$. The $K$-number for $R$ is $7$. The diagram for the monoid is a collapse of (i) in Figure \ref{fig:LOCDUAL} found in Section \ref{sec:DUAL}, where the link between $rar$ and $arar$ is collapsed.

\end{example}

\begin{example}\label{ex:2} Let $F$ be a field and let $R=F[x]/(x^3)$. Then the radical-annihilator monoid for $R$ satisfies the relations $a^2=1$, $r^2=r$, $rararar=rar$, $rara=rar$. 

The ring has four linearly ordered ideals. The $k$-numbers of the trivial ideals are $4$, and the $k$-numbers of the nontrivial ideals are $2$. The $K$-number of $R$ is $8$. The diagram for the monoid is precisely (i) in Figure \ref{fig:LOCDUAL} found in Section \ref{sec:DUAL}.

\end{example}

\begin{example}\label{ex:3} (the Clark example). Let $D$ be a discrete valuation ring and $Q$ be its field of fractions. Let $V$ be the $D$ module $Q/D$, and form the trivial extension $R=D\times V$. In more detail, we mean $(r,d)+(r',d'):=(r+r',d+d')$ and $(r,d)(r',d'):=(rr',rd'+dr')$. It is explained in (\cite{NYQFR} Example 6.6 pg 133) that the (infinitely many) ideals of $R$ are linearly ordered and that $R$ is a dual ring as defined in Section \ref{sec:DUAL}. 

There are two prime ideals: $P_1=\{0\}\times V$, and $P_2=M\times V$, where $M$ is the maximal ideal of $D$. By considering several representative types of ideals, we can show that $rar(I)=P_1$ for any ideal $I$ of $R$. Since $a(P_1)=P_1=r(P_1)$, $rar=arar=rara$, so this monoid is like the one in Example \ref{ex:1}. Let $I$ be an ideal strictly between $P_1$ and $P_2$. Then the elements of $\{1,a,r,ar,ra\}$ produce distinct results from $I$, and this is maximal, so the $k$-number of $R$ is $5$. Any ideal between zero and $P_1$ distinguishes $ra$ from $ara$ and $ara$ from $rar$, so the $K$-number for $R$ must be $7$. The radical-annihilator monoid of this ring is again (i) in Figure \ref{fig:LOCDUAL}.

\end{example}

In the context of the Kuratowski closure-complement problem, it is somewhat unexpected that there is a nontrivial relation between the two generators that limits the size of the monoid. Considering that the radical-annihilator relations are algebraically not so far from the closure-complement operations, it is interesting to ask if a radical-annihilator monoid must be finite or not.

\section{The monoid of a finite product of rings}\label{sec:PRODUCTS}

There is a very strong connection between the radical-annihilator monoids of rings and finite products of rings. This is all possible because of the following easily verified facts. Given and element $w$ of $M$ in terms of $r$ and $a$, we use the convention that $w_i$ is the corresponding element of $M_i$ written in terms of $r_i$ and $a_i$.

\begin{prop}\label{prop:products} In a ring product $R=\prod_{i=1}^n R_i$:
\begin{enumerate}
\item The ideals of $R$ are precisely those of the form $\prod_{i=1}^n I_i$ with $I_i\lhd R_i$.
\item In the notation of the last item, $a(I)=\prod_{i=1}^n a_i(I_i)$ where the $a_i$ are the annihilator maps in the $R_i$.
\item An ideal $\prod_{i=1}^n I_i$ is prime iff $I_j$ is prime in $R_j$ for some $j$ and $I_k=R_k$ for $k\neq j$. Consequently $r(I)=\prod_{i=1}^n r_i(I_i)$ where $r_i$ denotes the radical map in $R_i$.
\item If $w,v$ are two elements of $M$, then a relation $w=v$ holds iff $w_i=v_i$ for every $i$.
\item If $J=\prod_{i=1}^n J_i$ is another ideal, then $I\subseteq J$ iff $I_i\subseteq J_i$ for all $i$. Consequently, if $w, v$ are two elements of $M$ then $w\leq v$ iff $w_i\leq v_i$ for every $i$.

\end{enumerate}
\end{prop}

Suppose again that $R=\prod_{i=1}^n R_i$ where all radical and annihilator maps are defined, and we call $R$'s monoid $M$ and each $R_i$'s monoid $M_i$. The strong connection just observed suggests that we define a map from $M$ into $\prod_{i=1}^n M_i$ where $r\mapsto (r_1, r_2, \ldots, r_n)$, $a\mapsto(a_1, a_2, \ldots, a_n)$ and $1\mapsto (1_1, 1_2, \ldots, 1_n)$. This map is clearly unique and injective, but not usually surjective. The image of this map is a submonoid of $\prod_{i=1}^n M_i$, and we will denote it by $\bigodot_{i=1}^n M_i$. $\bigodot_{i=1}^n M_i$ is, by definition, the monoid of $\prod_{i=1}^n R_i$.

The list of observations in the proposition above imply that the relations and ordering of $\bigodot M_i$ can be recovered from known relations and orderings of the $M_i$. So, it is safe to study radical-annihilator monoids of ring direct summands of rings of interest. This will be helpful for both dual rings and zero-dimensional rings.

There are projections $\bigodot_{i=1}^n M_i\to M_j$ for each $j$. Considering this and the fact that $\bigodot_{i=1}^n M_i\subseteq \prod_{i=1}^n M_i$ , the following corollaries should be evident.

\begin{cor}
Denote the radical-annihilator monoid of $R_i$ by $M_i$. Then  $|\bigodot_{i=1}^n M_i|$ is bounded above by $\prod_{i=1}^n |M_i|$ and from below by $\max_{i=1}^n{|M_i|}$
\end{cor}

\begin{cor}\label{cor:samemonoid}If a finite collection of rings all share the same radical-annihilator monoid $M$, then the monoid of their product is $M$ as well.\end{cor}

The requirement that the collection of rings should be finite is indispensable: a consequence of the upcoming Lemma \ref{lem:STRONGREG} is that an infinite product of fields (which is a non-Noetherian strongly regular ring) must have monoid $\{1,a,a^2\}$ even though each of the factors of the product has monoid $\{1, a\}$.

The structures and orderings of monoid products in this article were generated with the assistance of Python software written by the author. The code is available on GitHub (\cite{pomonoidpy}).

\section{Radical-annihilator monoids of some classes of rings}\label{sec:TYPES}

\subsection{Dual rings}\label{sec:DUAL}
By virtue of their definition, commutative dual rings have radical-annihilator monoids governed by relations similar to those in the Kuratowski monoid. If $a_r$ and $a_\ell$ denote the right annihilator and left annihilator maps respectively in a noncommutative ring $R$, then $R$ is said to be \textbf{right dual} if $a_r(a_\ell(T))=T$ for every right ideal $T$, and \textbf{left dual} if $a_\ell(a_r(L))=L$ for every left ideal $L$. A \textbf{dual ring} is just a left and right dual ring, and clearly this is the same thing as $a^2=1$ when $R$ is commutative.

We could even consider noncommutative rings which are ``dual on ideals'' in the sense that $aa(I)=I$ for every two-sided ideal $I$. Requiring such rings to be semiprime in order to get an unambiguous annihilator map is an option, but at the very least it implies a small monoid: just $\{1,a\}$. It will be seen in Corollary \ref{cor:SPDUAL} that a semiprime dual ring is already semisimple (in the sense of rings characterized by the Artin-Wedderburn theorem), so there is little more to say in that case. For this reason, the sequel continues with commutative dual rings.

Naturally, the radical-annihilator monoid of a commutative dual ring is a quotient of the full Kuratowski monoid: since $a^2=1$ for commutative dual rings, the map determined by $c\mapsto a$ and $k\mapsto r$ and the first isomorphism theorem for monoids suffice to see this. It follows that its diagram must be a contraction of the diagram in Figure \ref{fig:KURA} satisfying $rar=rararar$. It turns out that the slightly stronger relation $rar=rarara$ holds in every commutative dual ring. After a journey past a theorem and several lemmas, we will establish this relation.

\begin{thm}\label{thm:RAR}
For any ring $R$ where the radical-annihilator monoid is defined, the relation $rar=rara$ holds in the monoid of $R$ iff $ar(\{0\})\subseteq r(\{0\})$. In this case, $rar$ is the constant map with value $r(\{0\})$.
\end{thm}
\begin{proof}
Suppose first that $ar(\{0\})\subseteq r(\{0\})$. For any ideal $I$, $r(I)\supseteq r(\{0\})$ by order-preserving property of $r$. Then using the order-reversing property of $a$ and the hypothesis, $ar(I)\subseteq ar(\{0\})\subseteq r(\{0\})$. Recalling that $r(J)=r(\{0\})$ for any ideal $J$ contained in $r(\{0\})$, we can conclude $rar(I)=r(\{0\})$. In particular, this is true with $I$ replaced by $a(I)$, so both $rar$ and $rara$ are the constant map with value $r(\{0\})$. 

On the other hand, if $rar=rara$, we have $r(\{0\})=ra(R)=rar(R)=rara(R)=rar(\{0\})\supseteq ar(\{0\})$.
\end{proof}

\begin{lem}\label{lem:DUALINTERSECTION}In a commutative dual ring, $a(\cap I_i)=\sum a(I_i)$ for any indexed set of ideals $I_i$.\end{lem}
\begin{proof}
By the order-reversing property of $a$, we have $a(\cap I_i)\supseteq a(I_j)$ for every $j$ in the index set, and it follows that $a(\cap I_i)\supseteq \sum a(I_i)$. 

If $x\in a(\sum a(I_i))$, then in particular $x\in aa(I_j)=I_j$ for all $j$. This shows $a(\sum a(I_i))\subseteq \cap I_i$. Applying $a$ to both sides, $a(\cap I_i)\subseteq aa(\sum a(I_i))\subseteq \sum a(I_i)$.

\end{proof}

\begin{lem}\label{lem:NILPRIMEa}For a commutative ring with minimal prime $P$, $a(P)$ is a nil ideal iff $a(P)\subseteq P$.\end{lem}
\begin{proof} First, if $a(P)$ is a nil ideal, then it is contained in the nilradical, and hence in all prime ideals, $P$ included.

Now suppose $a(P)\subseteq P$. If $Q$ is a minimal prime different from $P$, then we have immediately that $a(P)\subseteq Q$, for $Pa(P)\subseteq Q$, but $P\subseteq Q$ is impossible due to minimality of $Q$ and distinctness of $Q$ from $P$. At this point it has been shown $a(P)$ is contained in \emph{every} minimal prime ideal, and so is a nil ideal contained in the nilradical $N$.\end{proof}

The author is indebted to Keith A. Kearnes for pointing out the following lemma, which is the key to show that $a(N)\subseteq N$ in local dual rings which are not fields.

\begin{lem}\label{lem:KEARNES}In a commutative local dual ring which is not a field, $a(P)\subseteq P$ for every prime ideal $P$.\end{lem}
\begin{proof} There are two cases: $P$ is the maximal ideal and $P$ is nonmaximal. If $P$ is the maximal ideal and $R$ is not a field, $a(P)$ is a proper ideal of $R$ and is necessarily contained in $P$, the unique maximal ideal.

Instead if $P$ is not maximal, then we claim that $P=\cap\{I\mid P\subsetneq I\}$. If we label this intersection as $K$, then in the integral domain $R/P$, $K/P$ is either the zero ideal or else it is the minimal ideal of $R/P$. Since a domain with a minimal ideal is a field, $K/P$ being minimal would imply that $R/P$ is a field, but that means $P$ is maximal, and that is contrary to assumption. Hence $K/P$ is the zero ideal, that is, $K=P$.

Finally, notice that for each $I\supsetneq P$, the containment $Ia(I)\subseteq P$ yields that $a(I)\subseteq P$ since it is impossible for $I$ to be contained in $P$. Applying this to $a(P)=\sum_{P\subsetneq I}a(I)$, we see that the right hand side is contained entirely in $P$, and so $a(P)\subseteq P$.

\end{proof}

\begin{thm}\label{thm:dualRAR}Excepting fields, we have that $a(N)\subseteq N$ for any commutative local dual ring with nilradical $N$.\end{thm}
\begin{proof}In the local case, index the family of minimal prime ideals as $\{P_i\mid i\in I\}$. By Lemma \ref{lem:DUALINTERSECTION}, $a(N)=\sum_{i\in I}a(P_i)$. Combining Lemma \ref{lem:NILPRIMEa} and Lemma \ref{lem:KEARNES}, each $a(P_i)$ is a nil ideal, so their sum is contained in $N$. Thus $a(N)\subseteq N$.
\end{proof}

\begin{cor}
In any non-field commutative local dual ring, the annihilator of a prime ideal is a nil ideal.
\end{cor}

\begin{proof}If $P$ is any prime ideal, then $P\supseteq N$. Immediately it follows $a(P)\subseteq a(N)\subseteq N$, and so $a(P)$ is also nil.

Of course, the conclusion cannot hold when the local dual ring is a field, since then the only prime ideal $\{0\}$ has annihilator $R$ and is not nil.
\end{proof}

Beginning with the Kuratowski monoid (Figure \ref{fig:KURA}) with $r$'s and $a$'s substituted for $k$'s and $c$'s respectively, the additional relation provided by $rar=rara$ collapses to (i) in Figure \ref{fig:LOCDUAL}, and the dual ring in Exercise \ref{ex:2} substantiates that a ring with that diagram exists. 

\begin{thm}A commutative local dual ring $R$ has a radical-annihilator monoid which is one of the following two types: 

\begin{enumerate}[label=(\alph*)]
\item satisfying the relations $r=a^2=1$; (This is the case when $R$ is a field.)
\item satisfying relations $1=a^2$ and $rar=rara$.
\end{enumerate}

The first type has $K$-number 2, and the second type has $K$-number at most $8$. Furthermore, by computing the $\bigodot$ product of the monoids, the relation $rar=rarara$ holds in the radical-annihilator monoid of a (possibly nonlocal) commutative dual ring. The $K$-number for such a composite monoid is at most $10$.
\end{thm}
\begin{proof}
If $R$ is a field the analysis for case (a) should be apparent. If $R$ is not a field, then Theorem \ref{thm:dualRAR} says that Theorem \ref{thm:RAR} applies, and so $R$ is of the second type. All that remains is to prove the claims about (possibly nonlocal) commutative dual rings.

Decompose the dual ring into $\prod_{i=1}^n R_i$ where the $R_i$ are local dual rings, each having its respective radical-annihilator monoid $M_i$, so that $R$ has radical-annihilator monoid $\bigodot_{i=1}^n M_i$. By Proposition \ref{prop:products}.4, we may simplify by coalescing monoids of the same type. Therefore it suffices to consider the $\bigodot$ product of the two monoid types above. We'll compute with $a=(a_1, a_2)$ and $r=(r_1, r_2)$ where $a_1,r_1$ are the defining maps from the monoid of the field type, and $a_2,r_2$ are the defining maps from the monoid of the nonfield type.

Using the relations for the field monoid, $r_1a_1r_1a_1r_1a_1=a_1a_1a_1=a_1$ and similarly $r_1a_1r_1=a_1$, showing that this part of the $\bigodot$ product satisfies the identity $rar=rarara$.

Using the relations for the nonfield monoid, $(r_2a_2r_2a_2)r_2a_2=(r_2a_2r_2)r_2a_2=r_2a_2r_2a_2=r_2a_2r_2$, showing that the other part of the $\bigodot$ product also satsifies the identity $rar=rarara$.

Hence, the entire monoid of any commutative dual ring satisfies $rar=rarara$. By taking the direct product of two rings attaining the two monoid types, one easily sees a monoid obtaining the diagram in Figure \ref{fig:LOCDUAL} (ii) with $K$-number $10$.
\end{proof}

The relation $rar=rarara$ can be compared to $kck=kckckck$ in the Kuratowski monoid, and it implies $rar=rararar$, of course.

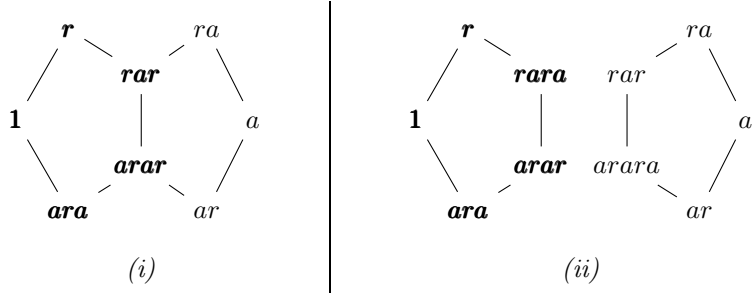
\begin{figure}[ht!]
\begin{center}
\begin{tikzpicture}[description/.style={fill=white,inner sep=2pt}]
\matrix (m) [matrix of math nodes, row sep=0.3em,
column sep=0.3em, text height=1.5ex, text depth=0.25ex, scale=0.33]
{  &\pmb{r}  &    &ra& & \\
   &   &\pmb{rar} &  & & \\
  \pmb{1}&   &    &  &a& \ \ \\
   &   &\pmb{arar}&  & & \\ 
   &\pmb{ara}&    &ar& & \\};
   \path[-] (m-1-2) edge (m-2-3)
            (m-1-2) edge (m-3-1)
            (m-3-1) edge (m-5-2)
			(m-2-3) edge (m-4-3)
            (m-1-4) edge (m-2-3)
            (m-4-3) edge (m-5-2)
            (m-4-3) edge (m-5-4)
            (m-3-5) edge (m-5-4)
			(m-1-4) edge (m-3-5)
;
\node at (-0.2,-2) {\textit{(i)}};
\end{tikzpicture}
\vrule{}
\qquad
\begin{tikzpicture}[description/.style={fill=white,inner sep=2pt}]
\matrix (m) [matrix of math nodes, row sep=0.3em,
column sep=0.3em, text height=1.5ex, text depth=0.25ex, scale=0.33]
{  &\pmb{r}  &    &     &ra& \\
   &   &\pmb{rara}&rar  &  & \\
  \pmb{1}&   &    &     &  &a\\
   &   &\pmb{arar}&arara&  & \\ 
   &\pmb{ara}&    &     &ar& \\};
   \path[-] (m-1-2) edge (m-2-3)
            (m-1-2) edge (m-3-1)
            (m-1-5) edge (m-2-4)
            (m-1-5) edge (m-3-6)
            (m-3-1) edge (m-5-2)
            (m-3-6) edge (m-5-5)
            (m-2-3) edge (m-4-3)
            (m-4-3) edge (m-5-2)
            (m-2-4) edge (m-4-4)
            (m-4-4) edge (m-5-5)
            
;
\node at (0,-2) {\textit{(ii)}};
\end{tikzpicture}
\caption{The largest radical-annihilator monoids for (i) a local dual ring and (ii) an arbitrary dual ring. }
\label{fig:LOCDUAL}
\end{center}
\end{figure}

Let us take stock of how many Kuratowski monoids are appearing as radical-annihilator monoids of commutative dual rings. All of Examples $\ref{ex:1}$, $\ref{ex:2}$, and $\ref{ex:3}$ are dual rings.

\begin{enumerate}[label=(\alph*)]
\item The full Kuratowski monoid is not attainable: all radical-annihilator monoids of commutative dual rings satisfy the relations for the Kuratowski monoid for an open unresolvable space as depicted in \cite{GardJack08} Figure 2.1. Taking the product of the Example \ref{ex:2} with a field produces the Kuratowski monoid of an open unresolvable space.

\item The diagram for the monoid of an open unresolvable and extremally disconnected space is attained by taking the product of ring in Example \ref{ex:1} with a field.

\item A field, of course, yields the same monoid as the Kuratowski monoid of a discrete space.

\item Clearly we cannot produce a radical-annihilator monoid which matches the full Kuratowski monoid of an extremally disconnected space because the monoid must also satisfy the relations for an open unresolvable space.

\item The radical-annihilator monoid cannot match the full Kuratowski monoid of a partition space. The relation in commutative dual rings ($rar=rarara$), together with the relation for partition spaces ($arar=r$) would imply in sequence that $ra=rar$ and $ara=arar=r$, so that it also has the defining relation of a discrete space.

\item Finally, of the two ``almost possible'' monoids mentioned in \cite{GardJack08} Figure 2.2, the commutative dual ring relation ($rar=rarara$) combines with the relation in the larger monoid ($r=rarar$) to generate $r=rara$, and the relations $r=rara=rarar$ produce the smaller monoid. We proceed to show that the smaller monoid is still not attainable. 

If $r(\{0\})=\{0\}$ (i.e. $R$ is semiprime) then $R$ is a finite product of fields and has monoid $\{1,a\}$, so we proceed assuming $r(\{0\})\neq \{0\}$. Then $ra(\{0\})=R$, and in succession $r(\{0\})\neq \{0\}$ implies $ar(\{0\})\neq R$ and $rar(\{0\})\neq R$. So $ra\neq rar$ and this last monoid is also not possible.
\end{enumerate}

\subsection{Semiprime rings}\label{sec:SEMIPRIME}

For this section, we consider semiprime rings which are possibly not commutative. As discussed in the introduction, the annihilator map is unambiguous. Some basic results already appearing in the literature point to a characterization of semiprime rings by their monoid.
  
\begin{lem}(\cite{LAMLMR} Lemma 11.40)
For an ideal $I$ in a semiprime ring $R$, $a(I)=\cap\{P\mid P \text{ prime and } I\not\subseteq P\}$.
\end{lem}

\begin{remark}\label{rem:SEMIPRIME}
The condition in the last lemma actually characterizes semiprime rings: $\{0\}=a(R)=\cap\{P\mid P \text{ prime and } R\not\subseteq P\}=N(R)$ in particular, so $R$ is semiprime.
 \end{remark}
 
\begin{thm}
Among rings for which the radical-annihilator monoid is defined, the following conditions are equivalent:
\begin{enumerate}
 \item $R$ is a semiprime ring;
 \item The radical-annihilator monoid is abelian; $(ra=ar)$
 \item Annihilator ideals are semiprime; $(ra=a)$
 \item $ar(I)= a(I)$ for every ideal $I$; $(ar=a)$
 \item $a^2$ dominates $r$ in the monoid order. $(r\leq a^2)$
\end{enumerate} 
 
 Consequently, the monoid of such a ring must be a submonoid of $\{1,a,a^2,r\}$.
 \end{thm}
 \begin{proof}
  First let $I$ be an ideal of a semiprime ring $R$. The preceding lemma shows that each annihilator is a semiprime ideal, so $ra=a$. Also, the prime ideals containing $I$ are exactly those containing $r(I)$, and so prime ideals not containing $I$ are exactly those not containing $r{I}$. Therefore the intersection of the primes not containing $I$ is $a(I)$ and $ar(I)$ at the same time. Thus $ar=a$. The two relations together imply that $ra=ar$.
  
  Conversely, if $ra=a$, then $\{0\}=a(R)=ra(R)=r(\{0\})$ implies $R$ is semiprime. If instead $ar=a$, then $ar(\{0\})=a(\{0\})=R$, but this implies that $r(\{0\})=\{0\}$ since nonzero ideals have proper annihilators in rings with identity. Finally if $ra=ar$, then $R=r(R)=ra(\{0\})=ar(\{0\})$, implying $r(\{0\})=\{0\}$ as in the last case.
  
If $r\leq a^2$, then $r(\{0\})\leq a^2(\{0\})=\{0\}$ shows $r(\{0\})=\{0\}$, so $R$ is semiprime. On the other hand, consider a prime $P$ which does not contain $a(I)$. Then $P$ must contain $I$. The set of primes not containing $a(I)$ is therefore a subset of those containing $I$. Then the intersection of primes containing $I$ is contained in the intersection of primes not containing $a(I)$, but this is saying that $r(I)\subseteq a^2(I)$.
  
  It is easy to compute that $\{1,a,a^2,r\}$ is the largest monoid generated by $r$ and $a$ under those relations, and this is obtained by Example \ref{ex:4}.
 \end{proof}

\begin{figure}[ht!]
\begin{center}
\begin{tikzpicture}[description/.style={fill=white,inner sep=2pt}]
\matrix (m) [matrix of math nodes, row sep=1.5em,
column sep=0.3em, text height=1.5ex, text depth=0.25ex]
{ \pmb{a^2}&\ \ & \\
  \pmb{r}  &\ \ &a\\
  \pmb{1}  &\ \ & \\};

\path[-] (m-1-1) edge (m-2-1)
         (m-2-1) edge (m-3-1)
;
\node at (0,-2) {\textit{(i)}};
\end{tikzpicture}
\vrule\qquad
\begin{tikzpicture}[description/.style={fill=white,inner sep=2pt}]
\matrix (m) [matrix of math nodes, row sep=1.5em,
column sep=0.3em, text height=1.5ex, text depth=0.25ex]
{ \pmb{a^2}&\ \ & &\\
     &\ \ &a& \ \ \\
  \pmb{1}  &\ \ & &\\};
\path[-] (m-1-1) edge (m-3-1)
;
\node at (-0.2,-2) {\textit{(ii)}};
\end{tikzpicture}
\caption{The largest radical annihilator monoids (i) of a semiprime ring having relations $r^2=r$; $a^3=a$; $ra=a=ar$ and (ii) of a fully semiprime ring additionally satisfying $r=1$}
\label{fig:SEMIPRIME}
\end{center}
\end{figure}
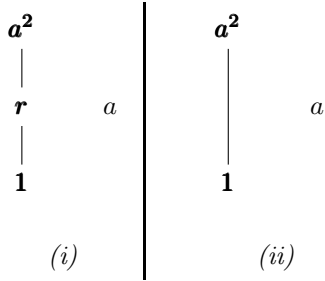

  \begin{example}\label{ex:4}
  Let $D=F[x]$ where $F$ is a field. For $M=(x^2)$, $r(M)=(x)$, $a(M)=\{0\}$, $aa(M)=D$ are all distinct, showing that the $k$- and $K$-numbers for the ring are both $4$.
  \end{example}
  
The last example shows that domains cannot be distinguished from other semiprime rings by their monoid and relations. We are compelled to talk about the elements of the monoid to determine the distinction. Clearly in a domain, $a$ maps all nonzero ideals to the zero ideal, and the zero ideal to the whole ring. Thus the image of $a$ has only two elements, and it is clear that this is the most trivial that $a$ can become. 

The following observation generalizes this to semiprime rings and distinguishes the class of simple rings within the class of prime rings.

\begin{prop}
A ring $R$ is prime iff the annihilator map $a$ is the map sending nonzero ideals to zero and the zero ideal to $R$. If $R$ is a prime ring, then $R$ is simple iff the relation $a^2=r$ holds in its monoid.
\end{prop}
\begin{proof}
For the first claim: let $R$ be prime. Then for every nonzero ideal $I$, $Ia(I)=\{0\}$ implies $a(I)=\{0\}$ by primeness of $\{0\}$. In the other direction, suppose $AB=\{0\}$. If $A$ is nonzero then $a(A)=\{0\}$ by hypothesis. Since $B\subseteq a(A)$, we have shown $B=\{0\}$, and that $R$ is prime.

For the second claim: let $R$ be prime and $M$ be a maximal ideal of $R$ and suppose $r=a^2$. Then $M=r(M)=a(a(M))$. Since $M$ is proper, $a(M)\neq\{0\}$. By primeness, $M=\{0\}$. Thus $R$ is simple. In the other direction, if $R$ is simple it is trivial to check (on the two ideals) that $a^2=r$.
\end{proof}

\begin{cor}
A commutative ring is a domain iff the annihilator map $a$ is the map sending nonzero ideals to zero and the zero ideal to $R$. A commutative domain is a field iff $a^2=r$ in its monoid.
\end{cor}

Here is the result alluded to in the introduction of Section \ref{sec:DUAL}.

\begin{lem}Let $R$ be a directly indecomposable semiprime ring in which no proper ideal is faithful. Then $R$ is a simple ring.\end{lem}

\begin{proof}Let $I\lhd R$. By Lemma 11.38 of \cite{LAMLMR} pp 334, the sum $K=I + a(I)$ is a direct sum. This ideal $K$ is faithful: to see this, note that  $a(K)=a(I)\cap aa(I)$, and that the square of this ideal is zero. In a semiprime ring, this implies the ideal is zero.

By hypothesis then, $K$ cannot be proper, so $R=I\oplus a(I)$. Since $R$ is directly indecomposable, either $I=\{0\}$ or $I=R$. This demonstrates $R$ is simple.
\end{proof}

\begin{cor}\label{cor:SPDUAL}If $R$ is a semiprime, dual ring, then $R$ is a semisimple Artinian ring with radical-annihilator monoid $\{1,a\}$.
\end{cor}
\begin{proof}
Since $R$ is dual it is semiperfect, and hence is a finite direct product of rings which are directly indecomposable rings, and each of these is dual and semiprime as well. Without loss of generality then, we additionally suppose $R$ is directly indecomposable as a ring and proceed to show such a ring is a simple Artinian ring.

Because $R$ is right dual, proper ideals have nonzero right annihilators, so the hypotheses of the preceding lemma apply, and $R$ is simple. A maximal left ideal must exist, and its annihilator must be a minimal right ideal, so the right socle of the ring is nonzero and is necessarily the entire ring since $R$ is simple. This concludes the proof that each indecomposable piece of the ring is simple Artinian, and so the original ring is a semisimple ring.
\end{proof}

Rings in which every ideal is a semiprime ideal have been called \textbf{fully semiprime rings}. In the spirit of the present article, we paraphrase this by saying $R$ is fully semiprime if and only if $r=1$. Fully semiprime rings are clearly always semiprime, and Figure \ref{fig:SEMIPRIME} shows that a semiprime ring satisfying $a^2=1$ is fully semiprime. This class of rings includes the important class of von Neumann regular rings. At this point, the reader is cautioned to remember that ``dual'' is strictly stronger than ``$a^2=1$'' for noncommutative rings. There exist simple non-Artinian von Neumann regular rings which satisfy $a^2=1$ trivially and yet they are not dual.

A proposition similar to \ref{cor:SPDUAL} holds for a special class of von Neumann regular rings called \textbf{strongly regular rings}, which are by one definition von Neumann regular rings in which idempotents commute with all elements of the ring. The distinguishing feature here is that von Neumann regular rings do not have to split up into finitely many indecomposable rings.

\begin{lem}\label{lem:STRONGREG}
Let $R$ be a strongly regular ring. Then $a^2=1$ in the monoid of $R$ iff $R$ is right Noetherian. When this is the case, $R$ is actually a finite product of division rings.\end{lem}
\begin{proof}
A finitely generated ideal of a von Neumann regular ring $R$ is of the form $eR$ for an idempotent $e$, and one easily computes that $a(eR)=(1-e)R$ when $e$ is central. Then it
follows that $a^2$ is the identity on finitely generated ideals. If $R$ is right Noetherian, then $aa$ is the identity on all ideals, so that the monoid is $\{1,a\}$.

It is well-known (\cite{LAMFC} p.66 Theorem 4.25) that if $R$ is not right Noetherian, then it is not semisimple. Since it is not semisimple, it has a proper essential right ideal (\cite{StenstromROQ} p.54 Lemma 2.1). Call this essential, proper right ideal $I$. Furthermore, right ideals are two-sided ideals in strongly regular rings (\cite{GDRLVNR} p.26 Theorem 3.2), so $I\lhd R$. We claim that $a(I)=\{0\}$. If it were not so, then there would be a nontrivial idempotent $e$ such that $eR\subseteq a(I)$, and then $(1-e)R \supseteq aa(I)\supseteq I$. But this is a problem since it implies that $(1-e)R$ is an essential ideal, but it cannot be an essential ideal since it intersects $eR$ trivially. Thus $a(I)=\{0\}$, and $aa(I)=R$, so that the monoid is $\{1,a,aa\}$.
\end{proof}

\subsection{Commutative local zero-dimensional rings}
Since a commutative local zero-dimensional ring has only a single prime ideal, it is relatively easy to deduce what the monoid looks like.

\begin{thm}\label{thm:RisRAA}
For a commutative local zero-dimensional ring $R$ with maximal ideal $M$, the following conditions are equivalent:
\begin{enumerate}
\item $a^2(M)=M$;
\item $a^2r=r$;
\item $ra^2=r$.
\end{enumerate}
\end{thm}
\begin{proof}
We show that the first item implies the following two, and then that the following two imply the first item. In several places we use the fact that $r(I)=M$ for every proper ideal $I$.

Suppose $a^2(M)=M$. Then $a^2r(I)=a^2(M)=M=r(I)$ for any proper ideal of $R$. Also, $a^2r(R)=r(R)$ always holds, so $a^2r=r$. Secondly, for any proper ideal $I$, $a^2(I)$ must also be proper since it is contained in $a^2(M)=M$. This is why $ra^2(I)=M=r(I)$ and $ra^2=r$.

We now look at $a^2(M)$ which can only be $M$ or $R$ by extensiveness of $a^2$ and maximality of $M$. If $a^2r=r$, then $R=a^2(M)=a^2r(M)=r(M)=M$ is a contradiction. Similarly, if $ra^2=r$ and $a^2(M)=R$, then $M=r(M)=r(R)=R$ is a contradiction. 
\end{proof}

\begin{thm}\label{ZDIM}
The monoid of a commutative local zero-dimensional ring is a contraction of one of three types of monoids:
\begin{enumerate}[label=(\alph*)]
\item a field having relations $r=a^2=1$;
\item having relations $rar=rara$ and $r=a^2r=ra^2$, with $k$-number at most $5$;
\item having relations $ar=ara$ and $a^2r=ra^2r$, with $k$-number at most $6$.
\end{enumerate}

The first monoid has a $K$ number of $2$, and the last two have $K$-number $9$.
\end{thm}

\begin{proof} Let $M$ be the maximal ideal. If $M=\{0\}$, we are looking at a field and the monoid is clearly $\{1,a\}$ with the relations $r=1$ and $a^2=1$. From now on, we assume $M\neq \{0\}$.

Suppose that $a(M)\neq \{0\}$. Since $M$ is a nil ideal and has a proper annihilator, Theorem \ref{thm:RAR} applies so that $rara=rar$ holds. If $a^2(M)=M$, then Theorem \ref{thm:RisRAA} applies, and this is a ring of type (b). It is easy to discover that the trivial ideals can only produce $R, M, \{0\}$ and $a(M)$. A nontrivial ideal $I$, on the other hand, can only produce $a(I), a^2(I), M, a(M)$ for a total of $5$ distinct sets.

Turning to the case when $a(M)=\{0\}$, we necessarily have $a^2(M)=R$. Since the range of $r$ is $\{R, M\}$, the map $ar$ is uniformly $\{0\}$. In particular, $ar(a(I))=\{0\}$ for all $I$, so $ara$ is also uniformly zero, and equal to $ar$. Now $a^2r$ is uniformly $R$, and nothing changes if $r$ is applied after $a^2r$, so it is equal to $a^2r$ as well.

Note that $ar$ and $ara$ are distinct for a ring of the type (b). For such a ring $ar(R)=\{0\}$, but $ara(R)=a(M)\neq \{0\}$. So, types (b) and (c) are distinct in general.

\end{proof}

Unfortunately, the shape of these monoids does not characterize $0$-dimensional rings. The ring in Example $\ref{ex:3}$ is a two-dimensional ring whose monoid is a contraction of type (b) above. 

Here are diagrams and examples for types (b) and (c).

\begin{figure}[ht!]
\begin{center}
\begin{tikzpicture}[description/.style={fill=white,inner sep=2pt}]
\matrix (m) [matrix of math nodes, row sep=0.5em,
column sep=0.3em, text height=1.5ex, text depth=0.25ex]
{  &   &\pmb{r}  &    &ra& & \\
   &   &   &\pmb{rar} &  & & \\
   &\pmb{a^2}&   &    &  &a& \ \ \\
   &   &   &\pmb{arar}&  & & \\ 
  \pmb{1}&   &\pmb{ara}&    &ar& & \\ };

\path[-] (m-1-3) edge (m-3-2)
		 (m-1-3) edge (m-2-4)
         (m-1-5) edge (m-2-4)
         (m-1-5) edge (m-3-6)
         (m-2-4) edge (m-4-4)
         (m-4-4) edge (m-5-3)
         (m-4-4) edge (m-5-5)
         (m-3-2) edge (m-5-3)
         (m-3-6) edge (m-5-5)
         (m-3-2) edge (m-5-1)  
;
\node at (0,-2.5) {\textit{(i)}};
\end{tikzpicture}
\vrule
\qquad
\begin{tikzpicture}[description/.style={fill=white,inner sep=2pt}]
\matrix (m) [matrix of math nodes, row sep=1.0em,
column sep=0.3em, text height=1.5ex, text depth=0.25ex]
{      &  &\pmb{a^2r}& \ \ &   \\
   ra^2&  &    &  &ra \\
   \pmb{a^2} &\pmb{r} &    &  &a  \\
       &\pmb{1} & rar&  &   \\
       &  &  \pmb{ar}&  &   \\ };

\path[-] (m-1-3) edge (m-2-1)
         (m-2-1) edge (m-3-1)
         (m-3-1) edge (m-4-2)
         (m-4-2) edge (m-5-3)
         (m-2-1) edge (m-3-2)
         (m-3-2) edge (m-4-3)
         (m-4-3) edge (m-5-3)
         (m-2-5) edge (m-3-5)
         (m-1-3) edge (m-2-5)
         (m-2-5) edge (m-4-3)
         (m-3-5) edge (m-5-3)
         (m-3-2) edge (m-4-2)      
;
\node at (0,-2.5) {\textit{(ii)}};

\end{tikzpicture}

\caption{The radical annihilator monoid of a commutative local zero-dimensional ring (i) with relations $r=a^2r=ra^2$ and $rar=rara$ and (ii)  and (ii) with relations $a^2r=ra^2r$ and $ar=ara$.}
\label{fig:ZDRtypes}
\end{center}

\end{figure}
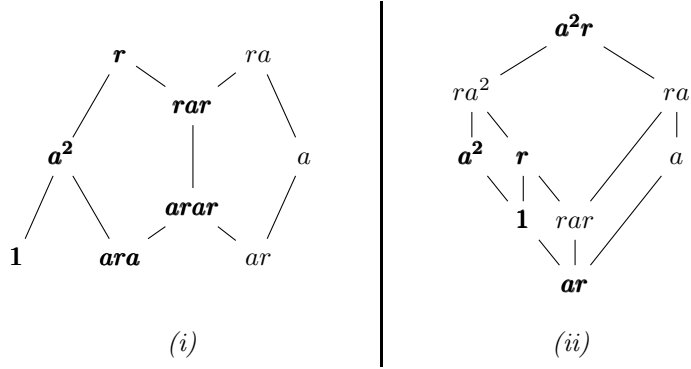

\begin{example} For a field $F$, the ring $F[x,y]/(x,y)^5$ is a ring of type (b) in the last theorem.   If $I=(x^2)$, then $a(I)=(x,y)^3$, $a^2(I)=(x,y)^2$, $ra(I)=(x,y)$, and $ara(I)=(x,y)^4$. This accounts for $5$ sets.
\end{example}
 
\begin{example}\label{ex:ZDIMC}
For an example of the latter type, let $R$ be the quotient of $\mathbb Q[x^{1/2}, x^{1/4}, x^{1/8},\ldots, y^{1/2}, y^{1/4}, y^{1/8}, \ldots]$ by the ideal of elements of the form $x^py^q$ where $p,q$ are dyadic rationals with sum $1$. It is apparent that the ideal $M$ generated by these fractional powers of $x$ and $y$ is nil since its generators are nilpotent. The quotient by this ideal is isomorphic to $\mathbb Q$, so it is the unique maximal ideal. Finally, it's easy to see that for any nonzero element $r$, one can always find an $n$ large enough such that $x^{2^{-n}}r\neq 0$, so $r$ does not annihilate $M$. Let $I=(x^{1/4})$. Then $a(I)=(\{x^{p}y^{q}\mid p+q\geq 3/4\})$ where $p,q$ are dyadic rationals. Now $aa(I)$ contains $(y^{1/4})$ which is not contained in $I$, so $a^2(I)\neq I$. We have $r(I)=ra(I)=ra^2(I)=M$, and the annihilator of any of these is $\{0\}$, whose annihilator is $R$, accounting for six sets.
\end{example}

By taking a local zero-dimensional ring for each of the types (a), (b), and (c) in the classification, we find the largest possible radical-annihilator monoid for a finite product of such rings. In particular, this gives us the largest monoids of commutative Artinian and perfect rings.

\begin{prop}
Suppose $M_1$, $M_2$, $M_3$ are monoids of types (a), (b), (c) respectively. Then
\begin{itemize}
\item $M_1\bigodot M_2$ has $11$ elements, relations $rara=rarar$ and $r=a^2r=ra^2$.
\item $M_1\bigodot M_3$ has $12$ elements, relations $ara=arar$, $a^2r=ra^2r$, and $ar=ara^2$.
\item $M_2\bigodot M_3$ has $13$ elements, relations $rar=rara$, $a^2r=ra^2r$, and $ar=ara^2$.
\item $M_1\bigodot M_2\bigodot M_3$ has $16$ elements, relations $ara^2=ar$, $ra^2r=a^2r$, $rara=rarar$.
\end{itemize}
\end{prop}

Figure \ref{fig:ZDR} diagrams the largest of these monoids.

\begin{figure}[ht!]
\begin{center}
\begin{tikzpicture}[description/.style={fill=white,inner sep=2pt}]
\matrix (m) [matrix of math nodes, row sep=0.5em,
column sep=0.3em, text height=1.5ex, text depth=0.25ex]
{  &   &    &      & \ \ &     &    &   \\  
   &   &\pmb{a^2r} &      & \ \ &     &a^2ra&   \\
   &ra^2&    &\pmb{a^2rara}& \ \ &a^2rar&    &ra \\
  \pmb{r}&   &    &      & \ \ &     &    &   \\
   &   &\pmb{rara}&      & \ \ &     &rar &   \\
   &\pmb{a^2} &    &\pmb{arar}  & \ \ &arara&    &a  \\
  \pmb{1}&   &\pmb{ara} &      & \ \ &     &ar  &   \\
   &   &    &      & \ \ &     &    &   \\};

\path[-] (m-2-3) edge (m-3-2)
         (m-2-3) edge (m-3-4)
         (m-3-2) edge (m-4-1)
         (m-3-2) edge (m-6-2)
         (m-4-1) edge (m-7-1)
         (m-4-1) edge (m-5-3)
         (m-6-2) edge (m-7-1)
         (m-6-2) edge (m-7-3)
         (m-3-4) edge (m-5-3)
         (m-5-3) edge (m-6-4)
         (m-6-4) edge (m-7-3)
         (m-2-7) edge (m-3-6)
         (m-2-7) edge (m-3-8)
         (m-3-6) edge (m-5-7)
         (m-3-8) edge (m-5-7)
         (m-3-8) edge (m-6-8)
         (m-5-7) edge (m-6-6)
         (m-6-6) edge (m-7-7)
         (m-6-8) edge (m-7-7)    
;
\end{tikzpicture}
\caption{The largest possible radical-annihilator monoid of a finite product of commutative local zero-dimensional rings. It is no mistake that $ra^2$ is not in bold since it is not idempotent: $(ra^2)^2=a^2r$.}
\label{fig:ZDR}
\end{center}
\end{figure}
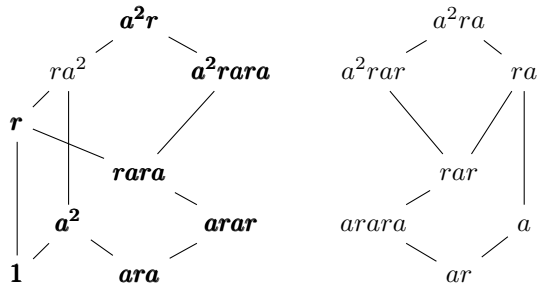

\subsection{Further comments on the results}
The types of maps in the radical-annihilator monoid depart from the types found in the Kuratowski monoid in very visible ways. For one thing, the monoid is not always finite. Furthermore, maps in the radical-annihilator monoid can be constant maps. We will also comment on sometimes surprising patterns of idempotency of the maps.

In \cite{GardJack08}, maps are ``even'' or ``odd'' based on evenness or oddness of the number of $c$'s that appear in the map. We choose not to adopt the same nomenclature for radical-annihilator monoids because the count becomes ill-defined and maps become both even and odd, as in the case of rings where $rara=rar$. 

Still, it is true that maps with an even number of $a$'s are order-preserving, and the ones which can be factorized with an odd number of $a$'s are order-reversing. When a map can be factorized both ways, it has to be both order-preserving and order-reversing.

\begin{prop}If $f$ is an order-preserving and order-reversing map on a poset with a greatest element or a least element, $f$ must be a constant map.\end{prop}

\begin{proof}If $R$ is the greatest element and $I$ is an arbitrary element of the poset, $I\subseteq R$ implies $f(I)=f(R)$ after applying both hypotheses of $f$. The proof in the case the poset has a least element is similar.\end{proof}

This also sheds some light on why the diagrams of some dual rings and some zero-dimensional rings are connected and not split into two disjoint pieces: the order-preserving and order-reversing maps can overlap.

For dual rings, the constant maps are $rar$ and $arar$ appearing in Figure $\ref{fig:LOCDUAL}$ (i), but the same maps are not constant in (ii) of the same figure. For commutative local zero-dimensional rings, the constant maps in Figure \ref{fig:ZDRtypes} (i) are $rar$ and $arar$, since $rar=rara$ and $arar=arara$. In (ii) of the same figure, $a^2r$, $rar$ and $ar$ are constant since $arar=ar$.

As for the idempotency of the operators, recall that the idempotent operators in diagrams have been boldfaced. The Kuratowski monoid is remarkably regular with this respect since all of its even operators are idempotent. The radical-annihilator monoid for a dual ring, which is just a collapse of the Kuratowski monoid, unsurprisingly has the same property. The radical-annihilator monoid for semiprime rings and type (i) local zero-dimensional rings follow suit.

But surprisingly, the local zero-dimensional rings of type (ii), and consequently the full diagram in \ref{fig:ZDR}, break this pattern via the oddball map $ra^2$. In Example \ref{ex:ZDIMC}, $ra^2(I)=M$, but $ra^2ra^2(I)=ra^2(M)=R$.

So far the only class of rings that do not satisfy $rar=rara$ have been simple rings and their finite products. Here is an example of a commutative nonfield ring which does not satisfy this relation. 

\begin{example}Let $F$ be a field and consider the ring $F[x,y]/(x^2,xy)$. The nilradical is the ideal $(x)$, and it is apparent that $rar(x)=ra(x)=r(x,y)=(x,y)$, but $rara(x)=rar(x,y)=ra(x,y)=r(x)=(x)$.

\end{example}

\section{The dualradical map}\label{sec:DUALRAD}

Following \cite{Henriksen1990}, the \textbf{hull} of an ideal $I$ is the set of prime ideals containing $I$ and is denoted by $h(I)$. The complement of this set in the set of all prime ideals of $R$, the set of prime ideals not containing $I$, is called the \textbf{hull complement of $I$} and denoted by $h^c(I)$. 

Taking $r(I)=\cap h(I)$ as inspiration, we create a map on ideals of $R$ called $d$ given by $d(I)=\cap h^c(I)$ and call it the \textbf{dualradical map}. As we saw in Remark \ref{rem:SEMIPRIME}, the relation $d=a$ characterizes semiprime rings. Unlike $a$, there are no ambiguity problems when defining $d$ in any ring.

Considering that the two hulls are related by complementation in $Spec(R)$, it is interesting to ask if $Spec(R)$ is somehow connected to the radical-annihilator monoid. It is well known that $Spec(R)$ is homeomorphic to the set of prime ideals of $Spec(R/r(\{0\}))$ under the hull-kernel topology -- the topology using hulls of ideals as closed sets. Since the monoid of semiprime rings is limited compared to the monoids of arbitrary rings, it is clear that the monoid and its relations do not carry information that is fine-grained enough. If any such characterizations do exist, they necessarily mention specific properties about $r$ and $a$ that go beyond their algebraic relations.

One such result appeared recently: Birkenmeier, Ghirati and Taherifar obtained a characterization of semiprime rings with extremally disconnected spectra. It was given in terms of the radical map's behavior on the lattice of ideals. The condition is that $r(I\cap J)=r(I)+r(J)$ for all pairs of ideals $I,J$ (\cite{Birken_2015} Theorem 4.4). This is doubly of interest since the class of extremally disconnected spaces is one which the Kuratowski monoid can discriminate.

Further introspection into specific properties of $r$ and $a$ will undoubtedly be fruitful; however, we now return to the current program of seeing what the monoid and its relations tell us about the ring.

Earlier in Section \ref{sec:SEMIPRIME} it was shown that a ring is semiprime iff $r\leq a^2$. Using the hull and hull complement notation, we generalize that conclusion and record additional properties of $d$.

\begin{thm}\label{thm:DUALPROPS}
Let $R$ be a ring in which $r$ and $a$ are defined, and let $d$ denote the dualradical map. Then the following is true:
\begin{enumerate}
\item $d$ is an order-reversing map on the ideals of $R$;
\item $ra\leq d$ and $r\leq da$;
\item $r\leq d^2$;
\item $d^3=d$;
\item $rd=d=dr$.
\end{enumerate}

\end{thm}
\begin{proof}
($d$ is order-reversing) If $I, J$ are ideals of $R$ and $I\subseteq J$, then any prime not containing $I$ cannot contain $J$. Therefore $h^c(I)\subseteq h^c(J)$, and $d(J)\subseteq d(I)$.

($ra\leq d$ and $r\leq da$) If $P$ is prime and $I$ is an ideal and $I\nsubseteq P$, then $a(I)\subseteq P$. This says that $h^c(I)\subseteq h(a(I))$. Taking complements in the set of prime ideals, we also have that $h(I)\supseteq h^c(a(I))$. In the former case, intersecting the two sets yields $d(I)\supseteq ra(I)$, and in the latter case $r(I)\subseteq da(I)$.

($r\leq d^2$) If $I\nsubseteq P$, then $d(I)\subseteq P$. The contrapositive asserts that $h^c(d(I))\subseteq h(I)$. Taking intersections, $d^2(I)\supseteq r(I)$.

($rd=d$) The output of $d$ is always semiprime, and $r$ acts like an identity on semiprime ideals.

($d^3=d$) On one hand, $d^2\geq r\geq 1$, and since $d$ is order-reversing, $d^3\leq d$. On the other hand, $d^2\geq r$ applied to $d(I)$ implies $d^3(I)\geq rd(I)=d(I)$ for all ideals $I$. Thus $d^3\geq d$.

($dr=d$) By order-reversal, $r\geq 1$ implies $dr\leq d$. From $d^2\geq r$, it follows that $d=d^3\leq dr$. 
\end{proof}

Call the monoid generated by $\{r,d\}$ the radical-dualradical monoid.

\begin{cor}
For any ring $R$, the radical-dualradical monoid of $R$ is a collapse of the four element monoid $\{1, r, d, d^2\}$, which is subject to the relations $r^2=r$, $d^3=d$, and $rd=d=dr$. The diagram of the largest such monoid is the same as Figure \ref{fig:SEMIPRIME} (i) with $d$'s substituted for $a$'s.
\end{cor}

\begin{cor}
If $d^2=1$ in the radical-dualradical monoid of $R$, then $R$ is fully semiprime, hence $a=d$. The monoid then coincides with the collapse of the radical-annihilator monoid of a semiprime ring which consists of two disconnected points: the identity and $d$ ($=a$).
\end{cor}

\begin{prop}
If the relation $d^2=r$ holds in a radical-dualradical monoid of $R$, then $R$ is zero-dimensional and $d(I)\neq r(\{0\})$ for any proper ideal $I$.
\end{prop}

\begin{proof}
Suppose $d^2=r$ and $M$ is a maximal and nonminimal prime ideal. Then $d(M)=r(\{0\})$ since no minimal prime ideal contains $M$. It would follow that $d^2(M)=R\neq M$. So we see that $d^2=r$ necessitates that every maximal prime is also a minimal prime. Continuing with $I\lhd R$, if $d(I)=r(\{0\})$, then $d^2(I)=R$, but since $I$ is proper, $r(I)\neq R$. 
\end{proof}

\begin{example}
Of course, any semiprime ring which has a maximal radical-annihilator monoid automatically has a maximal radical-dualradical monoid. This is an example of a ring which has a maximal size radical-dualradical monoid but is not semiprime. Let $R=\mathbb Z (+) \mathbb Z$ be the trivial extension of $\mathbb Z$ by $\mathbb Z$. The operations are $(r, m)+(s, n)=(r+s, m+n)$ and $(r,m)(s,n)=(rs, rn+ms)$, just as in the construction in Example \ref{ex:3}. Let $I=(4)(+)\mathbb Z$. Then $d(I)=0(+)\mathbb Z$ and $d^2(I)=R$ and $r(I)=(2)(+)\mathbb Z$. Since $R$ has nonzero nilpotent elements (all of $0(+)\mathbb Z$), the ring is not semiprime or fully semiprime.
\end{example}

\section{Future work}
Several questions suggested by the foregoing work are gathered here.

\begin{qn} Is there an example of a ring with an infinite radical-annihilator monoid?
\end{qn}

\begin{qn}Is there a noncommutative ring satisfying $a^2=1$ whose radical-annihilator monoid matches the full Kuratowski monoid?\end{qn}

The class of commutative dual rings provided monoids that were most parallel with the original Kuratowski monoid. Surprisingly, there appeared a further relation beyond the special one arising from the order-reversing and order-preserving properties of the maps. If there is any hope of finding a ring satisfying $a^2=1$ that has a radical-annihilator monoid of size $14$, then it must be noncommutative.

\begin{qn}Can any conclusions be drawn about the monoids of commutative local rings?\end{qn}

If a classification exists for commutative local rings, then via the product theorem we have completely understood semiperfect rings. This would necessarily subsume the foregoing results on dual rings and local zero-dimensional rings.

\begin{qn}Can any conclusions be drawn about the monoids of rings whose prime ideals are linearly ordered?\end{qn}

Rings in which the prime ideals are linearly ordered are sometimes called \textbf{pseudo-valuation rings}. Because the radical map would be relatively simply behaved, it may be possible to generalize what this article has presented on zero-dimensional rings. In a pseudo-valuation ring of finite Krull dimension, each application of the radical map takes the ideals into a finite linearly ordered set, and one would like to believe that the associated monoid might be finite.

It is also possible to exchange the prime radical for other radicals. An idea along these lines is to use the Jacobson radical of an ideal in place of the radical of an ideal. This would seem to yield a ``coarser'' monoid since the resulting map dominates the radical map we have used in this article.

\section*{Acknowledgments}
The author wishes to thank the referee for his careful reading and numerous helpful recommendations. Thanks also to Greg Oman and Manfred von Willich for their time spent reading and giving feedback. Finally, thanks to Keith A. Kearnes for bringing Lemma \ref{lem:KEARNES} to light.

\bibliographystyle{ieeetr}
\bibliography{./Masterbib}

\end{document}